\documentclass[reqno,12pt]{amsart}
\usepackage[margin=1in,footskip=0.25in]{geometry}
\usepackage{amsmath,amssymb,amsfonts,hyperref,
graphicx,tikz,mathrsfs,amsaddr,blindtext,comment}
\hypersetup{hidelinks}

\makeatletter
\@namedef{subjclassname@2020}{2020 Mathematics Subject Classification}
\makeatother

\newtheorem{theorem}{Theorem}[section]
\newtheorem{lemma}[theorem]{Lemma}

\theoremstyle{definition}

\newtheorem{definition}[theorem]{Definition}

\newtheorem{corollary}[theorem]{Corollary}

\newtheorem{remark}[theorem]{Remark}

\numberwithin{equation}{section}

\def\C{\mathbb{C}}

\def\A{\mathcal{A}}

\def\I{\mathcal{I}}

\def \h{\mathcal{H}}\def \G{\mathcal{G}}
\def \T{\mathcal{T}}
\def \F{\mathcal{F}}

\def\x{\mathbf{x}}

\def \la{\lambda}

\begin{document}

\title[A hypergraph Heilmann--Lieb theorem]
{A hypergraph Heilmann--Lieb theorem}

\date{\today}


\author[Jiang-Chao Wan, Yi Wang, Yi-Zheng Fan]{Jiang-Chao Wan, Yi Wang, Yi-Zheng Fan}
\address{Center for Pure Mathematics, School of Mathematical Sciences,\\ Anhui University, Hefei 230601, Anhui, China}
\email{wanjc@stu.ahu.edu.cn, wangy@ahu.edu.cn, fanyz@ahu.edu.cn}
\thanks{{\it Corresponding author.} Yi Wang}
\thanks{{\it Funding.} Supported by the National Natural Science Foundation of China (No. 12171002)}

\begin{abstract}
The Heilmann--Lieb theorem is a fundamental theorem in algebraic combinatorics which provides a characterization of the distribution of the zeros of matching polynomials of graphs.
In this paper, we establish a hypergraph Heilmann--Lieb theorem as follows.
Let $\h$ be a connected $k$-graph with maximum degree ${\Delta}\geq 2$
and let $\mu(\h, x)$ be  its matching polynomial.
We show that the zeros (with multiplicities) of $\mu(\h, x)$ are invariant under a rotation of an angle $2\pi/{\ell}$ in the complex plane for some positive integer $\ell$ and $k$ is the maximum integer with this property.
We further prove that the maximum modulus $\lambda(\h)$ of all the zeros of $\mu(\h, x)$ is a simple root of $\mu(\h, x)$ and satisfies
$$\Delta^{\frac{1}{  k}} \leq \lambda(\h)< \frac{k}{k-1}\big((k-1)(\Delta-1)\big)^{\frac{1}{  k}}.$$
To achieve these, we prove that $\mu(\h, x)$
divides the matching polynomial of the $k$-walk-tree of $\h$,
which generalizes a classical result due to Godsil
from graphs to hypergraphs.
\end{abstract}

\maketitle

\section{Introduction}\label{section1}

The Heilmann--Lieb theorem~\cite{Heilmann} is a fundamental theorem in algebraic combinatorics which provides a characterization of the distribution of the zeros of matching polynomials of graphs.
To state it, let us recall the definition of the matching polynomial.
Given an $n$-vertex graph $G$, a {\em matching} in $G$ is a subset of its edges such that not two share a common vertex.
Write $p(G,r)$ for the number of matchings in $G$ consisting of $r$ edges.
In their celebrated paper~\cite{Heilmann}, Heilmann and Lieb defined the {\it matching polynomial} of $G$ to be the polynomial
 $$
\mu(G,x)=\sum_{r\geq 0}(-1)^rp(G,r) x^{n-2r}.
 $$

\begin{theorem}[Heilmann--Lieb~\cite{Heilmann}]
\label{Heilmanntheorem}
Let $G$ be a graph with maximum degree $\Delta(G) \geq 2$.
Then the zeros (with multiplicities) of $\mu(G,x)$ are symmetrically distributed about the origin
and lie in the interval
$\big(-2\sqrt{ {\Delta(G)} -1}, 2\sqrt{ {\Delta(G)} -1}\big)$.
\end{theorem}

The Heilmann--Lieb theorem has many impressive applications 
including spectral graph theory~\cite{Farrell,LovaszPe,Marcus}, combinatorics~\cite{Chen,Godsil4,KahnKim}, and statistical physics~\cite{Heilmann1,Heilmann}.
We refer readers to the textbooks~\cite{Godsil, Lovasz} for more background and history on matching polynomial theory, and see~\cite{Makowskya,Shi} for related graph polynomials.

The bound $2\sqrt{ {\Delta(G)} -1}$ for zeros of the matching polynomial of graphs in Theorem~\ref{Heilmanntheorem} is closely related to the second largest eigenvalues of graphs. If $G$ is a $d$-regular graph, then $d$ is always the largest adjacency eigenvalue of $G$, called the trivial eigenvalue of $G$.  The well-known Alon--Boppana Theorem~\cite{Alon} states that for all $d\geq2$ and $\varepsilon> 0$, there are
only finitely many $d$-regular graphs whose second largest eigenvalue is at most $2\sqrt{d-1}-\varepsilon$.
In addition,  Friedman~\cite{Friedman2008}
proved that for every $\varepsilon> 0$, with high probability,  random
$d$-regular graphs have the second largest eigenvalue smaller than $2\sqrt{d-1}+\varepsilon$, which was  conjectured by Alon~\cite{Alon}.  Motivated by the above results, Lubotzky, Phillips, and Sarnark~\cite{LPS} introduced the concept of {\it Ramanujan} graphs: $d$-regular graphs whose non-trivial eigenvalues are between $-2\sqrt{d -1}$ and $ 2\sqrt{d -1}$. It plays an important role in the study of the linear expander of graphs~\cite{Hoory}.
Based on Theorem~\ref{Heilmanntheorem},  Marcus, Spielman, and Srivastava~\cite{Marcus} showed that  there are infinitely many bipartite Ramanujan graphs  by  the breakthough technique called  interlacing families.
See~\cite{Hall,Marcus} for more details and related topics.

A {\it $k$-uniform hypergraph}  $\h=(V(\h), E(\h))$ consists of a  vertex set $V(\h)$ and an edge set $E(\h)$,
where each $e\in E(\h)$ is a $k$-element subset of $V(\h)$. In this paper, we use the term  ``$k$-graph'' for the case of $k$-uniform hypergraphs with $k \geq 2$, and the term ``graph'' exclusively for $k=2$.
A {\em matching} in $\h$ is a set of vertex-disjoint edges,
and we denote by $p(\h,r)$ the number of matchings in $\h$ consisting of $r$ edges.
Recently, to investigate the spectral radius of adjacency tensor of $k$-trees, Su, Kang, Li, and Shan~\cite{Suejc} introduced the following {\it matching polynomial} of a $k$-graph $\h$:
 $$
\mu(\h,x)=\sum_{r\geq 0}(-1)^rp(\h,r) x^{|V(\h)|-kr},
 $$
which is a minor adjustment based on the definition of a matching polynomial introduced by Zhang, Kang, Shan and Bai~\cite{Zhang}, and Clark and Cooper~\cite{ClarkCooper}.
However, most of their results~\cite{ClarkCooper,Suejc,Zhang} focus on the spectra of the adjacency tensor of $k$-trees but not on the properties of the matching polynomial itself.
This prompts us to explore more useful and interesting properties of the matching polynomial of $k$-graphs.

Inspired by the above classical works,
the purpose of this paper is to establish a hypergraph Heilmann--Lieb theorem.
To state it, we need to introduce more notation.
A real polynomial $f(x)$ is called {\it $\ell$-symmetric} if
\begin{equation}
\label{zerosymmetricdefi}f(x)=x^tg(x^\ell)
\end{equation}
for some nonnegative integer $t$ and some real polynomial $g(x)$.
In other words, $f(x)$ is $\ell$-symmetric if and only if its zeros remain invariant under a rotation of an angle $2\pi/\ell$ on the complex plane.
The maximum number $\ell$ such that~\eqref{zerosymmetricdefi} holds is called the  {\it cyclic index} of $f(x)$.
Let $\lambda(\h)$ denote the maximum modulus  of all zeros of $\mu(\h, x)$.
Clearly, Theorem \ref{Heilmanntheorem} provides that for a graph $G$  with $\Delta(G) \geq 2$, the cyclic index of $\mu(G, x)$ is $2$ and $\lambda(G)\leq 2\sqrt{ {\Delta(G)} -1}$.

We are now ready to present the main result of this paper,
which provides a characterization of the distribution of the zeros of matching polynomials of $k$-graphs.
In particular, 
it implies that when $k\geq 3$ the matching polynomial of a $k$-graph must contain a nonreal zero.

\begin{theorem}\label{maintheorem22}
Let $\h$ be a connected $k$-graph with maximum degree ${\Delta}\geq  2$.
Then the cyclic index of $\mu(\h, x)$ is $k$.
Moreover, the maximum modulus $\lambda(\h)$ of all the zeros of $\mu(\h, x)$ is a simple root of $\mu(\h, x)$ and satisfies
	\begin{equation}\label{mr}
	\Delta^{1\over k}\leq \lambda(\h)< \frac{k}{k-1}\big((k-1)(\Delta-1)\big)^{1\over k}.
	\end{equation}
\end{theorem}

The second eigenvalue of hypergraphs was introduced by Friedman and Wigderson~\cite{Friedman, FriedmanC95}. Lenz and Mubayi~\cite{LM} showed that a hypergraph satisfies some quasirandom properties if and only if it has a small second eigenvalue. In 2019, Li and Mohar~\cite{LiMohar} generalized the Alon--Boppana Theorem to $k$-graphs, and showed that for every finite $d$-regular  $k$-graph $\h$ on $n$ vertices, the second  eigenvalue of $\h$ is at least
$$\frac{k}{k-1}\big((k-1)(d-1)\big)^{1\over k}-o_n(1),$$
where $o_n(1)$ is a quantity that tends to zero for every fixed $d$ as $n\rightarrow \infty$.
Similar to the important application of Theorem~\ref{Heilmanntheorem} in Ramanujan graphs, Theorem~\ref{maintheorem22} is expected to be useful for extending Ramanujan graphs to hypergraphs.

As an application of Theorem~\ref{maintheorem22}, we obtain a new and short proof of the following result due to Friedman~\cite{Friedman} (see also~\cite{LiMohar} by Li and Mohar).

\begin{theorem}[\cite{Friedman,LiMohar}] \label{Lupperbound}
Let $\T$ be a $k$-tree with maximum degree $\mathnormal{\Delta}\geq  2$,
and let $\rho(\T)$ be the spectral radius of the adjacency tensor of $\T$.
Then
$$\rho(\T)<\frac{k}{k-1}\big((k-1)(\Delta-1)\big)^{1\over k}.$$
\end{theorem}

The rest of this paper is organized as follows.
In the next section we introduce some preliminary notation and results that will be used later.
In Section~\ref{Pathtrees}, we show that $\mu(\h, x)$
divides the matching polynomial of the $k$-walk-tree of $\h$.
In Section~\ref{distributionzeroskg}, we investigate  the distribution of the zeros of the matching polynomial
and complete the proofs of Theorems~\ref{maintheorem22} and~\ref{Lupperbound}.
Finally, we conclude this paper with some further discussion and  questions in Section~\ref{ConcludingRemarks}.

\section{Preliminaries}\label{Preliminaries}


\subsection{Notation}


Two $k$-graphs $\G$ and $\h$ are called {\it isomorphic}, denoted by $\G\backsimeq \h$,
if there exists a bijection $\theta:V(\G)\mapsto  V(\h)$ such that $ \{v_1, \ldots, v_k\}\in E(\G)$ if and only if $ \{\theta(v_1),  \ldots, \theta(v_k)\}\in E(\h)$.
We say that $\G$ is a {\it subgraph} of $\h$ if $V(\G)\subseteq V(\h)$ and $E(\G)\subseteq E(\h)$, and $\G$ is a {\it proper subgraph} of $\h$ if $\G$ is a subgraph of $\h$ and $\G\neq \h$.

An alternating sequence $p=(v_0,e_1,v_1, \ldots, e_\ell, v_\ell)$ of vertices and edges in $\h$ is called a {\it path} in $\h$ if the vertices and edges are distinct and ${v_{i-1}, v_i}\in e_i$ for $i =1, \ldots, \ell$.
The sequence $p$ is called a {\it cycle} in $\h$ if $p$ is a path in $\h$ with the additional condition $v_0= v_\ell$.
A $k$-graph is called a {\it $k$-forest} if it is acyclic,
and we say that a $k$-forest $\F$ is a {\it subforest} of $\h$ if $\F$ is also a subgraph of $\h$.
A $k$-graph $\h$ is {\it connected} if each pair of vertices of $\h$ are connected by a path,
and is a {\it $k$-uniform hypertree} (or simply {\it $k$-tree}) if $\h$ is both connected and acyclic.

Let $v$ be a vertex of $\h$.
Denote by $N_\h(v)$ the set of all vertices of $\h$ adjacent to $v$ and by $E_\h(v)$ the set of all edges of $\h$ incident to $v$.
The degree of $v$ is defined as $|E_\h(v)|$ and is denoted by $d_\h(v)$.
The maximum degree and minimum degree of the vertices of $\h$ is  denoted by $\Delta(\h)$ and $\delta(\h)$, respectively.
For a subset $W$ of $V(\h)$, let  $\h[W]$  denote the subgraph of $\h$ induced by $W$, i.e., $V(\h[W])=W$ and $E(\h[W])=\{e\in E(\h):e \subseteq W\}$.
For convenience, we simply write $\h-W$ instead of  $\h[V(\h)\setminus W]$,   write $\h-v$ for $\h-\{v\}$, and use $\h-e$  to denote  $\h-\{v_1, \ldots, v_k\}$ where $e=\{v_1, \ldots, v_k\}$ is an edge.

The following lemma provides some fundamental properties of the
matching polynomial.

\begin{lemma}[\cite{Suejc}]\label{basiclemma}
Let $\G$ and $\h$ be two vertex-disjoint $k$-graphs.
The following assertions hold.
\begin{enumerate}

\item $\mu(\G\oplus \h, x)=\mu(\G, x)\mu(\h, x)$, where $\G\oplus \h$ denotes the disjoint union of $\G$ and $\h$.

\item  For every vertex  $u\in V(\h)$,
$\mu(\h, x)=x\mu(\h-u, x)-\sum_{e\in E_\h(u)}\mu(\h-e, x).$

\item $\frac{d}{dx}\mu(\h, x)= \sum_{v\in V(\h)}\mu(\h-v, x).$

\end{enumerate}
\end{lemma}

\subsection{The characteristic polynomials of $k$-trees}

A real {\it tensor} (also called \emph{hypermatrix}) $\A=(a_{i_{1}  \ldots i_{k}})$ of order $k$ and dimension $n$ refers to a
  multi-dimensional array with entries $a_{i_{1}\ldots i_{k}}\in \mathbb{R}$
  for all $i_{j}\in [n]:=\{1,\ldots,n\}$ and $j\in [k]$.
Clearly, if $k=2$, then $\A$ is a square matrix of dimension $n$.
Let $\mathcal{I}=(\iota_{i_1\ldots i_k})$ be the {\it identity tensor} of order $k$ and dimension $n$, that is, $\iota_{i_{1} \ldots i_{k}}=1$ if
 $i_{1}=\cdots=i_{k} \in [n]$ and $ \iota_{i_{1} \ldots i_{k}}=0$ otherwise.

Let $\A=(a_{i_{1} \ldots i_{k}})$ be a tensor of order $k$ and dimension $n$.
For a vector $\x=(\x_1,\ldots,\x_n)^\top\in \mathbb{C}^{n}$, denote by $\x^{[k]}=(\x_1^{k},\ldots,\x_n^{k})^\top$ and let $\A \x^{k-1}$ be a vector in $\mathbb{C}^n$ whose $i$th component is 
\begin{align*}
(\A \x^{k-1})_i & =\sum_{i_{2},\ldots,i_{k}\in [n]}a_{ii_{2}\ldots i_{k}}\x_{i_{2}}\cdots \x_{i_k}.
\end{align*}
In 2005, Lim \cite{Lim} and Qi \cite{Qi} independently introduced the eigenvalues of tensors.
For some $\lambda \in \mathbb{C}$, if the polynomial system 
$$\A \x^{k-1}=\lambda \x^{[k-1]},$$ has a solution $\x \in \mathbb{C}^{n}\backslash \{0\}$,
then $\lambda $ is called an \emph{eigenvalue} of $\A$ and $\x$ is an \emph{eigenvector} of $\A$ associated with $\lambda$.

The \emph{determinant} of $\A$, denoted by $\det \A$, is defined as the resultant of the polynomial system $\A  \x^{k-1}$ \cite{Ha}
and the \emph{characteristic polynomial} $\phi_\A(x)$ of $\A$ is defined as $\det(x \I-\A)$~\cite{Qi}.
It is proved in~\cite[Theorem 1(a)]{Qi} that $\la$ is an eigenvalue of $\A$ if and only if it is a root of $\phi_\A(x)$.

 Let $\h$ be a $k$-graph on $n$ vertices $v_1,\ldots,v_n$.
The {\it adjacency tensor} \cite{CD} of $\h$ is defined as $\mathcal{A}(\h)=(a_{i_{1}\ldots i_{k}})$,
a tensor of order $k$ and dimension $n$,
where
\[a_{i_{1}i_{2}\ldots i_{k}}=\left\{
 \begin{array}{ll}
\frac{1}{(k-1)!}, &  \mbox{if~} \{v_{i_{1}},\ldots,v_{i_{k}}\} \in E(\h);\\
  0, & \mbox{otherwise}.
  \end{array}\right.
\]
In this paper, the eigenvalues of a $k$-graph $\h$ always refer to those of its adjacency tensor.
The {\it spectral radius of $\h$} is defined as
\[
\rho(\h)=\max\{|\lambda|: \lambda \mbox{ is an eigenvalue of } \A(\h) \},
\]
which is exactly the spectral radius of $\A(\h)$.


\begin{lemma}[\cite{KF}]\label{PerronFrobeniuslemma}
Let $\h$ be a connected $k$-graph.
If $\mathcal{G}$ is a subgraph $\mathcal{G}$ of $\h$, then
$\rho(\mathcal{G})\leq \rho(\h)$,
where the equality holds if and only if $\mathcal{G}=\h $.
\end{lemma}

Mowshowitz \cite{Mowshowitz} and independently Lov\'asz and Pelik\'an \cite{LovaszPe} proved that the characteristic polynomial of a tree coincides with its matching polynomial.
Inspired by this classical result,
Zhang, Kang, Shan, and Bai \cite{Zhang} obtained the eigenvalues with certain restrictions of a $k$-tree by its matching polynomial.
Subsequently, Clark and Cooper \cite{ClarkCooper} characterized all eigenvalues of a $k$-tree by the matching polynomials of its subhypertrees.
Recently, Li, Su, and Fallat \cite{Li} determined
the characteristic polynomial of the adjacency tensor of a $k$-tree
by the matching polynomials of its sub-hypertrees.
Here we only list two required results and refer their papers \cite{ClarkCooper,Li,Zhang} for the complete story.

\begin{theorem}
\label{dividesrelationthm}
Let $\T$ be a $k$-tree with adjacency tensor $\A(\T)$.
Then the following assertions holds.
\begin{enumerate}

\item~{\em (Corollary 3.2 \cite{Zhang}).}
     The largest real root of $\mu(\T,x)$ is equal to the spectral radius of $\A(\T)$.

\item~{\em (\cite{LovaszPe,Mowshowitz}, Corollary 5.6 \cite{Li}).} $\mu(\T,x)$ divides the characteristic polynomial of $\A(\T)$.

\end{enumerate}
\end{theorem}

\section{The $k$-walk-tree}\label{Pathtrees}

 The well-known path tree (also called Godsil's tree) of a graph, introduced by Godsil~\cite{Godsil2}, is considered as one of the most important and useful tools in matching polynomial theory.
For a graph $G$ and a vertex $u\in V(G)$,
the {\it path tree $T(G, u)$} is a tree which has vertices as the paths in $G$ starting at $u$, where two such paths
are adjacent if one is a maximal proper subpath of the other.
Godsil~\cite[Theorem 2.5]{Godsil2} established the following important theorem which has many applications in combinatorics~\cite{Godsil3,Godsil4,KahnKim}.

\begin{theorem}[Godsil \cite{Godsil2}]\label{Godsiltreethm}
Let $G$ be a connected graph with a vertex $u\in V(G)$.
Then we have
\begin{equation}\label{Godsiltreethmeq}
\frac{\mu(G-u,x)}{\mu(G, x)}
= \frac{\mu(T(G,u)-u,x)}{\mu(T(G,u),x)},
\end{equation}
and $\mu(G, x)$ divides $\mu(T(G,u),x)$.
\end{theorem}

To refute a conjecture by Kahn and Kim~\cite{KahnKim} regarding the random matchings of $k$-graphs,
Lee~\cite{Lee} introduced the concept of $k$-walk-tree, a hypergraph analog of the path tree.
The aim of this section is to prove that
$\mu(\h, x)$
divides the matching polynomial of the $k$-walk-tree of $\h$,
which will be utilized to prove Theorem~\ref{maintheorem22}.
We begin with the definition of the $k$-walk-tree by a recursive construction described in~\cite[Observation 3.4]{Lee}, which is equivalent to the original definition in~\cite[Definition 3.3]{Lee}.

\begin{definition}\label{defwalktreeequivalent}
Let $\h$ be a $k$-graph with a vertex $u\in V(\h)$ and a linear ordering $\prec$ on $V(\h).$
Suppose that $e_1,\ldots, e_t$ are all edges containing $u$ in $\h$ and the vertices in
$e_i = \{u, u_{(e_i, 1)}, \dots, u_{(e_i, k-1)}\}$ satisfy $u_{(e_i, 1)} \prec \cdots \prec u_{(e_i, k-1)}$
for every $i\in [t]$.
The {\it $k$-walk-tree} $\T(\h,\prec,u)$ of $\h$ rooted at $u\in V(\h)$ with respect to $\prec$ is defined to be the $k$-tree obtained from the collection of disjoint union of $k$-trees
      $$\Big\{\cup_{j = 1}^{k-1} \T\big(\h - \{u, u_{(e_i, 1)}, \dots, u_{(e_i, j-1)}\}, \prec,u_{(e_i, j)}\big): i\in [t]\Big\}$$
by adding an edge $\{u, (u_{(e_i, 1)}), \dots, (u_{(e_i, k-1)})\}$ for each $i\in [t]$,
where $(u_{(e_i, j)})$ denotes the root of
$\T\left(\h-\{v, u_{(e_i, 1)}, \dots, u_{(e_i, j-1)}\}, \prec, u_{(e_i, j)}\right)$
for each $i\in [t]$ and $j\in [k-1].$
\end{definition}

\begin{remark}
%
%
For a given labeled $k$-graph,
we observe that its
$k$-walk-tree depends not only on the choice of the root vertex but also on the linear ordering imposed on its vertex set.
Figure~\ref{figure2} illustrates two $k$-walk-trees of the $k$-graph $\mathcal{X}$ described in Figure~\ref{figure1}, both having the same root $u$. The difference between these trees arises from the choice of different linear orderings in $V(\mathcal{X})$.
%
\end{remark}

\begin{figure}[htb]
\centering
\begin{tikzpicture}

\node at (0,0){};
    \coordinate (u) at (3,1.2);

    \coordinate (a) at (1.2,0);
    \coordinate (b) at (4.8,0);

    \coordinate (z) at (3,-1.2);

    \coordinate (x) at (0 ,1.2);
    \coordinate (y) at (0,-1.2);

    \filldraw[fill=blue!30] (a) -- (x) -- (y)-- cycle ;
    \draw[fill=green!30] (u) -- (a) -- (b) -- cycle;
    \draw[fill=yellow!30] (a) -- (b) -- (z) -- cycle;

     \fill[black!80, draw=black, thick] (u) circle (2pt) node[black, above ] {$u$};
    \fill[black!80, draw=black, thick] (x) circle (2pt) node[black, above] {$x$};

    \fill[black!80, draw=black, thick] (a) circle (2pt) node[black, below ] {$a$};
    \fill[black!80, draw=black, thick] (b) circle (2pt) node[black, right] {$b$};

    \fill[black!80, draw=black, thick] (y) circle (2pt) node[black, below] {$y$};
    \fill[black!80, draw=black, thick] (z) circle (2pt) node[black, below] {$z$};

\node at (3.3,0.5)[left] {$e_1$};
\node at (3.3,-0.5)[left] {$e_2$};
\node at (0.8, 0 )[left] {$e_3$};


\node at (9.5,0 ){\begin{tabular}{|c|c|c|c|c|}

\hline

\hline

$V(\mathcal{X})=\{u,a,b,x,y,z\}$ \\

\hline

$E(\mathcal{X})=\{e_1,e_2,e_3\}$  \\

\hline
$e_1=\{u,a,b\}$  \\

\hline
$e_2=\{z,a,b\}$  \\

\hline
$e_3=\{a,x,y\}$  \\

\hline

Order $\prec_a: u \prec_a a \prec_a b \prec_a x \prec_a y \prec_a z$ \\

\hline

Order $\prec_b: u \prec_b b \prec_b a \prec_b x \prec_b y \prec_b z$ \\

\hline
\end{tabular}

};
\end{tikzpicture}
\caption{The $k$-graph $\mathcal{X}$.}\label{figure1}
\end{figure}

\begin{figure}[htb]
\centering
\begin{tikzpicture}

\node at (0,0){};
    \coordinate (u) at (2,1.2);

    \coordinate (a) at (0.2,0);
    \coordinate (b) at (3.8,0);

    \coordinate (x) at (-1.8,-1.8);
    \coordinate (y) at (-0.5,-1.8);

     \coordinate (s) at (0.9,-1.8);
    \coordinate (t) at (2.2,-1.8);

    \filldraw[fill=yellow!30] (a) -- (x) -- (y)-- cycle ;
    \draw[fill=green!30] (u) -- (a) -- (b) -- cycle;
    \draw[fill=blue!30] (a) -- (s) -- (t) -- cycle;

     \fill[black!80, draw=black, thick] (u) circle (2pt) node[black, above ] {$u$};

     \fill[black!80, draw=black, thick] (a) circle (2pt) node[black, left ] {$(u_{(e_1,a)})$};
    \fill[black!80, draw=black, thick] (b) circle (2pt) node[black, right] {$(u_{(e_1,b)})$};

    \fill[black!80, draw=black, thick] (x) circle (2pt) node[black, below ] {$(a_{(e_2,b)})$};
    \fill[black!80, draw=black, thick] (y) circle (2pt) node[black, below] {$(a_{(e_2,z)})$};

     \fill[black!80, draw=black, thick] (s) circle (2pt) node[black, below] {$(a_{(e_3,x)})$};
    \fill[black!80, draw=black, thick] (t) circle (2pt) node[black, below] {$(a_{(e_3,y)})$};


    \coordinate (ub) at (9,1.2);

    \coordinate (ab) at (7.2,0);
    \coordinate (bb) at (10.8,0);

    \coordinate (xb) at (6.4,-1.8);
    \coordinate (yb) at (8,-1.8);

     \coordinate (sb) at (10,-1.8);
    \coordinate (tb) at (11.6,-1.8);

   \coordinate (lb) at (9.2,-3.6);
    \coordinate (rb) at (10.8,-3.6);

    \draw[fill=green!30] (ub) -- (ab) -- (bb) -- cycle;

    \filldraw[fill=blue!30] (ab) -- (xb) -- (yb)-- cycle ;

    \filldraw[fill=yellow!30] (bb) -- (sb) -- (tb)-- cycle ;

     \filldraw[fill=blue!30] (sb) -- (lb) -- (rb)-- cycle ;


     \fill[black!80, draw=black, thick] (ub) circle (2pt) node[black, above ] {$u$};

     \fill[black!80, draw=black, thick] (ab) circle (2pt) node[black, left ] {$(u_{(e_1,a)})$};
    \fill[black!80, draw=black, thick] (bb) circle (2pt) node[black, right] {$(u_{(e_1,b)})$};

    \fill[black!80, draw=black, thick] (xb) circle (2pt) node[black, below] {$(a_{(e_3,x)})$};
      \fill[black!80, draw=black, thick] (yb) circle (2pt) node[black, below] {$(a_{(e_3,y)})$};

    \fill[black!80, draw=black, thick] (sb) circle (2pt) node[black, left] {$(b_{(e_2,a)})$};
      \fill[black!80, draw=black, thick] (tb) circle (2pt) node[black, right] {$(b_{(e_2,z)})$};

   \fill[black!80, draw=black, thick] (lb) circle (2pt) node[black, left] {$(a_{(e_3,x)})$};
      \fill[black!80, draw=black, thick] (rb) circle (2pt) node[black, right] {$(a_{(e_3,y)})$};

 \node at (2,-4) [below] {The $k$-walk-tree $\T(\h,\prec_a,u)$};

 \node at (9,-4) [below] {The $k$-walk-tree $\T(\h,\prec_b,u)$};


\end{tikzpicture}
\caption{Two $k$-walk-trees of the $k$-graph $\mathcal{X}$ rooted at $u$.}\label{figure2}
\end{figure}

\begin{theorem}[Theorem~3.5~\cite{KahnKim}]\label{treethm1}
Let $\h$ be a $k$-graph with a vertex $u\in V(\h)$ and a linear ordering $\prec$ on $V(\h).$
Then we have
$$
\frac{\mu(\h-u,x)}{\mu(\h, x)}
= \frac{\mu(\T(\h,\prec,u)-u,x) }{\mu(\T(\h,\prec,u),x)}.
$$
\end{theorem}

We are now ready to prove the main result of this section, which plays an important role in the proof of Theorem~\ref{maintheorem22}.

\begin{theorem}\label{d2}
Let $\h$ be a connected $k$-graph with a vertex $u\in V(\h)$ and a linear ordering $\prec$ on $V(\h)$.
Then for every vertex $u\in V(\h)$,
there exists a proper subforest $\F$ of $\T(\h,\prec,u)$ such that $$\mu(\h,x)=\frac{ \mu(\T(\h,\prec,u),x)}{\mu(\F,x)},$$
and hence that $\mu(\h,x)$ divides $\mu(\T(\h,\prec,u),x)$.
\end{theorem}
\begin{proof}
We prove the statement by induction on $|V(\h)|$.
If $|V(\h)|=k$, then $\h$ is a $k$-tree with one edge,
and the statement is trivial in this case.
Assume $|V(\h)|>k$.
By Theorem~\ref{treethm1}, we have
\begin{equation}\label{eqr1}
{\mu(\T(\h,\prec,u),x)}
= {\mu(\h, x)}\frac{\mu(\T(\h,\prec,u)-u,x) }   {\mu(\h-u,x)}.
\end{equation}
Thus, to prove the statement, it suffices to prove that there exists a subforest $\F$ of $\T(\h,\prec,u)$ such that the second factor on the right-hand side of~\eqref{eqr1} is the matching polynomial of $\F$.

Assume that $e_1,\ldots, e_t$ are all edges containing $u$ in $\h$ and the vertices in $e_i = \{u, u_{(e_i, 1)}, \dots, u_{(e_i, k-1)}\}$ satisfy $u_{(e_i, 1)} \prec \cdots \prec u_{(e_i, k-1)}$ for every $i\in [t]$.
Combining the definition of the $k$-walk-tree and Lemma~\ref{basiclemma}(1), we obtain
\begin{equation}\label{divideseq11}
{\mu(\T(\h,\prec,u)-u,x)}
=\prod_{i\in [t],\,j\in [k-1] }
 \mu\Big(  \T\big(\h - \{u, u_{(e_i, 1)}, \dots, u_{(e_i, j-1)}\}, \prec,u_{(e_i, j)}\big) \Big).
\end{equation}
Let $\h_1,\ldots,\h_s$ be the  components of $\h-u$.
By Lemma~\ref{basiclemma}(1), we get
\begin{equation}\label{divideseq22}
\mu(\h-u,x )=\prod_{i=1}^s \mu(\h_i,x).
\end{equation}
Denote by $N_\h(u)$ the set of all vertices adjacent to $u$ in $\h$.
For each $i=1, \ldots, s$,
let $u_i\in N_\h(u)\cap V(\h_i)$ be the unique vertex such that
$u_i \prec w$ for every $w\in N_\h(u) \cap V(\h_i)$.
Observe that there exists $a_i\in [t]$ such that $u_i\in e_{a_i}$.
Here, the edges $e_{a_i}$, $i=1, \ldots, s$, may be repeatedly selected.
Note that $e_{a_i}=\{u, u_{(e_{a_i}, 1)}, \dots, u_{(e_{a_i}, k-1)}\}$,
so there exists $b_i\in [k-1]$ such that $u_i=u_{(e_{a_i}, b_i)}$.
As $\h_1,\ldots,\h_s$ are different components of $\h-u$,
 the vertices $u_{(e_{a_i}, b_i)}$, $i=1, \ldots, s$, are all distinct.
By the choice of $u_i$ and $u_{(e_{a_i}, b_i)}$,
one may check that
$$\T\big(\h - \{u, u_{(e_{a_i}, 1)}, \dots, u_{(e_{a_i}, b_i-1)}\}, \prec,u_{(e_{a_i}, b_i)} \big)=:\T_i$$
is the $k$-walk-tree of $\h_i$ rooted at $u_{(e_{a_i}, b_i)}$. 
By the induction hypothesis,
for each $i=1, \ldots, s$,
there exists a proper subforest $\F_i$ of $\T_i$ such that
\begin{equation}\label{divideseq33}
\mu(\F_i,x) = \frac{\mu(\T_i, x)}{\mu(\h_i,x)}.
\end{equation}
Combining \eqref{divideseq11}, \eqref{divideseq22}, and \eqref{divideseq33}, we deduce that
\begin{align} \nonumber
\frac{\mu(\T(\h,\prec,u)-u,x) }   {\mu(\h-u,x)}
= & \frac{\prod_{i\in [t],\, j\in [k-1] }
 \mu\Big(  \T\big(\h - \{u, u_{(e_ i, 1)}, \dots, u_{(e_i, j-1)}\}, \prec,u_{(e_i, j)}\big),x \Big)
 }   {\prod_{i=1}^s \mu(\h_i,x)}\\ \nonumber
=& \frac{\mu(\G,x) \Big( \prod_{i=1}^s  \mu(\T_i,x) \Big) }   {\prod_{i=1}^s \mu(\h_i,x)} \\ \nonumber
=& \frac{\mu(\G,x) \Big( \prod_{i=1}^s  \mu(\h_i,x) \mu(\F_i,x)\Big) }   {\prod_{i=1}^s \mu(\h_i,x)}
\\ \label{divideseq44}
=& \mu(\G,x) \prod_{i=1}^s \mu(\F_i,x),
\end{align}
where
$$\G=  \bigoplus_{i\in [t],\, j\in [k-1],
\atop  (i,j)\neq (a_r, b_r) {\rm\,for\,every\,} r\in[s]}
\T\big(\h - \{u, u_{(e_i, 1)}, \dots, u_{(e_i, j-1)}\}, \prec,u_{(e_i, j)}\big).$$
Recall that $\F_i$ is a proper subforest of $\T_i$ for each $i=1, \ldots, s$,
so one may check that $\G \oplus \left( \bigoplus_{i=1}^s  \F_i \right)$ is a proper subforest of $\T(\h,\prec,u)$, which is the required subforest $\F$.
The statement follows from this fact, \eqref{eqr1} and \eqref{divideseq44}.
\end{proof}

\section{The distribution of the zeros of the matching polynomial}\label{distributionzeroskg}

This section is devoted to studying the distribution of the zeros of the matching polynomial.
In particular, we complete the proofs of Theorems~\ref{maintheorem22} and~\ref{Lupperbound}.

\subsection{The cyclic index of the matching polynomial}
In this subsection, we prove that the maximum modulus $\lambda(\h)$ of all the zeros of $\mu(\h, x)$ is a simple root of $\mu(\h, x)$ and the cyclic index of $\mu(\h, x)$ is exactly equal to $k$.
We begin with the following lemma,
which implies that the largest real root $\widehat{\lambda}(\F)$ of a $k$-forest $\F$ is equal to ${\lambda}(\F)$.




\begin{lemma}\label{ktreethreeequal}
For a $k$-forest $\F$, we have
$$
\lambda(\F)=\widehat{\lambda}(\F)=\rho(\F).
$$
\end{lemma}
\begin{proof}
By Lemma \ref{basiclemma}(1),
it suffices to prove that the statement holds for all $k$-trees.
Let  $\T$ be a $k$-tree.
Observe that $\rho(\T)=\widehat{\lambda}(\T)\leq \lambda(\T)$ by Theorem~\ref{dividesrelationthm}(1).
On the other hand, if $\lambda$ is a zero of $\mu(\T, x)$ such that
$|\lambda|=\lambda (\T)$, then $\lambda$ is an eigenvalue of $\A(\T)$  by Theorem~\ref{dividesrelationthm}(2).
By the definition of spectral radius, we get $|\lambda|=\lambda (\T)\leq \rho(\T)$. The result follows.
\end{proof}

\begin{theorem}\label{maintheorem11}
Let $\h$ be a connected $k$-graph with a linear ordering $\prec$ on $V(\h).$
Then for every $u \in V(\h)$,
$\mu(\h,x)$ divides the characteristic polynomial of
the adjacency tensor of the $k$-walk-tree $\T(\h,\prec,u)$.
Moreover, $\lambda(\h)$ is a simple root of $\mu(\h,x)$ and \begin{equation}\label{threeequality}
\lambda(\h)
=\lambda(\T(\h,\prec,u))=\rho(\T(\h,\prec,u)).
\end{equation}
\end{theorem}

\begin{proof}
 The first statement immediately follows from Theorem~\ref{dividesrelationthm}(2) and Theorem \ref{d2}.
We next prove that $\lambda(\h)$ is a root of $\mu(\h,x)$ and
\eqref{threeequality} holds.
By Theorem \ref{d2}, there exists a proper subforest $\F$ of $\T(\h,\prec,u)$ such that
\begin{align}\label{eqsubsection331}
\mu(\T(\h,\prec,u),x) = \mu(\h,x)\mu(\F,x),
\end{align}
which implies that
\begin{align}\label{eqsubsection3322}
\lambda(\T(\h,\prec,u))=\max\big\{\lambda(\h), \lambda(\F)\big\}.
\end{align}
Since $\F$ is a proper subforest of $\T(\h,\prec,u)$, we have $\rho(\F)<\rho(\T(\h,\prec,u))$ by Lemma \ref{PerronFrobeniuslemma}.
By Lemma~\ref{ktreethreeequal}, we have
\begin{align}\label{eqsubsecion33233}
\lambda(\F)=\widehat{\lambda}(\F)=\rho(\F)
<\rho(\T(\h,\prec,u))=
\lambda(\T(\h,\prec,u))=\widehat{\lambda}(\T(\h,\prec,u)).
\end{align}
Combining \eqref{eqsubsection331}, \eqref{eqsubsection3322} and
\eqref{eqsubsecion33233}, we
derive that $\lambda(\h)$ equals to $\lambda(\T(\h,\prec,u))$ and is a root of $\mu(\h,x)$, and \eqref{threeequality} follows, as desired.

It remains to prove that the root $\lambda(\h)$ is simple.
 We first claim that $\lambda(\h)>\lambda(\h-v)$ for each $v\in V(\h)$.
Let $\h_1,\ldots,\h_s$ be the components of $\h-v$.
Without loss of generality, we may assume
$\lambda(\h_1) =\lambda(\h-v).$
Using~\eqref{threeequality}, we have $\lambda(\h)=\rho(\T(\h,\prec,v))$ and $\lambda(\h-v)=\lambda(\h_1)=\rho(\T(\h_1,\prec,v_1))$,
where $v_1\in N_\h(v)\cap V(\h_1)$ is the unique vertex such that
$v_1 \prec w$ for every $w\in N_\h(v)\cap V(\h_1)$.
Observe that $\T(\h_1,\prec,v_1)$ is a proper subtree of $\T(\h,v)$.
By Lemma~\ref{PerronFrobeniuslemma},
$\rho(\T(\h,\prec,v))>\rho(\T(\h_1,\prec,v_1))$,
which implies $\lambda(\h)>\lambda(\h-v)$, as desired.
Note that the leading coefficient of $\sum_{v\in V(\h)}\mu(\h-v, x)$ is positive. It follows from above claim that
$\sum_{v\in V(\h)}\mu(\h-v, x)$ is positive whenever $x\geq \lambda(\h)$.
Therefore, $\lambda(\h)$ is not a root of
$\frac{d}{dx}\mu(\h, x)$
by Lemma \ref{basiclemma}(3), so the root $\lambda(\h)$ of $\mu(\h,x)$ is simple.
The proof is completed.
\end{proof}




As an application of Theorem~\ref{maintheorem11},
we next determine the cyclic index of the matching polynomial of $k$-graphs.

\begin{theorem}\label{zerosymmetricmp}
Let $\h$ be a connected $k$-graph.
Then the cyclic index of $\mu(\h,  x)$ is equal to~$k$.
\end{theorem}
\begin{proof}
Denote by $c$ the cyclic index of $\mu(\h,  x)$.
For the $k$-th root of unity $\xi$, one can check that
$$\mu(\h,\xi x)=\sum_{r\geq 0}(-1)^rp(\h,r) (\xi x)^{|V(\h)|-kr}= \xi ^{|V(\h)|} \mu(\h,x).$$
This implies that $\mu(\h,  x)$ is  $k$-symmetric, so we get $k\leq c$.

Since $\lambda(\h)$ is a simple root of $\mu(\h,  x)$ by Theorem  \ref{maintheorem11} and $\mu(\h,  x)$ is  $c$-symmetric,
we get that
$\lambda(\h) e^{\mathbf{i}\frac{2\pi j}{c}},
j = 0, 1,\ldots, c-1$,
are zeros of $\mu(\h,  x)$.
By Theorem  \ref{maintheorem11}, they are eigenvalues of a $k$-walk-tree $\T(\h,\prec,u)$ of $\h$ with modulus $\rho(\T(\h,\prec,u))$.
Let $d$ be the cyclic index of the characteristic polynomial of $\A(\T(\h,\prec,u))$.
Theorem 2.6 and Eq.~(2.7) in~\cite{FanTrans} imply that $\T(\h,\prec,u)$ has exactly $d$ distinct eigenvalues with modulus $\rho(\T(\h,\prec,u))$,
and Corollary 4.3 in~\cite{FanTrans} says that $d| k$.
We therefore derive that $c\leq d\leq k$, so we have $c=k$.
The result follows.
\end{proof}

\begin{remark}\label{remark3.2222}
From the proof of Theorem \ref{zerosymmetricmp}, we also get that $\mu(\h,  x)$ has exactly $k$ distinct zeros with modulus $\lambda(\h)$ and they are equally distributed on complex plane,  that is, they are $\lambda(\h) e^{\frac{2\pi j }{k}\mathbf{i}}, j = 0, 1,\ldots, k-1.$ We therefore conclude that $\mu(\h,  x)$ is $\ell$-symmetric if and only if $\ell$ divides $k$.
\end{remark}

\subsection{The largest zero of the matching polynomial}
In this subsection, we present lower and upper bounds for $\lambda(\h)$ and give the proofs of Theorems~\ref{maintheorem22} and~\ref{Lupperbound}.

\begin{lemma}\label{subpathtree}
Let $\h$ be a connected $k$-graph.
If $\mathcal{G}$ is a subgraph of $\h$, then
$\lambda(\mathcal{G})\leq \lambda(\h)$,
where the equality holds if and only if $\mathcal{G}=\h$.
\end{lemma}
\begin{proof}
Without loss of generality, we may assume that $\mathcal{G}$ is connected.
Otherwise, we can get the result by considering the components of $\mathcal{G}$ and using Lemma~\ref{basiclemma}(1).
Let $\prec$ be a linear ordering on $V(\h)$, 
and let $u\in V(\mathcal{G})$ be the unique vertex such that $u \prec w$ for every $w\in V(\mathcal{G})$.
Since $\mathcal{G}$ is a subgraph of $\h$ containing $u$, $\T(\mathcal{G},\prec,u)$ is a subgraph of $\T(\h,\prec,u)$.
	By~\eqref{threeequality} and Lemma \ref{PerronFrobeniuslemma}, we have
	$$\lambda(\mathcal{G})=\rho(\T(\mathcal{G},\prec,u))\leq \rho(\T(\h,\prec,u))=\lambda(\h),$$
where the equality holds if and only if $\T(\mathcal{G},\prec,u) = \T(\h,\prec,u)$.
Observe that $\T(\mathcal{G},\prec,u)=\T(\h,\prec,u)$ if and only if
$\mathcal{G}=\h$. The result follows.
\end{proof}

\begin{corollary}\label{coro22}
Let $\h$ be a connected $k$-graph with maximum degree ${\Delta}$.
Then $\lambda(\h)\geq \Delta^{1\over k}$, where the
equality holds if and only if all the edges of $\h$  share a common vertex.
\end{corollary}
\begin{proof}
Denote by $\mathcal{S}_\Delta$ the $k$-star with maximum degree $ {\Delta}$, that is,
the $k$-tree consisting of $\Delta$ edges sharing a common vertex.
Clearly,
$$\mu(\mathcal{S}_\Delta,x)
=x^{(k-1)\Delta -k+1}(x^k-\Delta),$$
and thus   $\lambda(\mathcal{S}_\Delta)=\Delta^{1\over k}$.

Let $u$ be a vertex of $\h$ with $d_\h(u)=\Delta$, and let $\prec$ be a linear ordering on $V(\h).$
Then $\T(\h,\prec,u)$ contains $\mathcal{S}_\Delta$ as a subtree.
By \eqref{threeequality} and Lemma~\ref{subpathtree}, we have
	$$\lambda(\h)=\lambda (\T(\h,\prec,u)) \geq \lambda(\mathcal{S}_\Delta)=\Delta^{1\over k}$$
	with equality holds  if and only if $\T(\h,\prec,u)=\mathcal{S}_\Delta.$
	If all the edges of $\h$ share a common vertex $u$, then $\T(\h,\prec,u)=\mathcal{S}_\Delta$ by the definition of $k$-walk-tree.
	Conversely, if $\T(\h,\prec,u)=\mathcal{S}_\Delta$, then $\h$ has exactly $\Delta$ edges as $|E(\h)|\leq |E(\T(\h,\prec,u))|=\Delta$ and $\Delta(\h)=\Delta$,
	which implies that all the edges of $\h$ have a common vertex.
\end{proof}

To establish the upper bound of $\lambda(\h)$,
we need the following auxiliary lemma.

\begin{lemma}\label{lemma}
Let $\h$ be a  connected $k$-graph with maximum degree $\Delta$  and let $\xi\geq \max\{\Delta,2\}$ be an integer.
If $u\in V(\h)$ and $d_\h(u)< \xi$, then
\begin{align}\nonumber
\frac{\mu(\h, x)}{\mu(\h-u,x)}> { \big((k-1)(\xi-1)\big)^{ 1\over k } }
\end{align}
whenever $x\geq \frac{k}{k-1}\big((k-1)( \xi-1)\big)^{1\over k}$.
\end{lemma}
\begin{proof}
We prove the statement by induction on $n=|V(\h)|$.
In this proof, we always assume that $x\geq \frac{k}{k-1}\big((k-1)( \xi-1)\big)^{1\over k}$.
If $n=k$,  then $\h$ is the $k$-graph consisting of a single edge,
and hence that
\begin{align} \nonumber
\frac{\mu(\h, x)}{\mu(\h-u,x)}
= \frac{  x^k-1 }{x^{k-1}}
=x-\frac{ 1 }{x^{k-1}}
>& \big((k-1)( \xi-1)\big)^{1\over k},
\end{align}
where the inequality follows from the calculation:
 \begin{align}  \nonumber
x-\frac{ 1 }{x^{k-1}}
\geq &  \frac{k}{k-1}\big((k-1)( \xi-1)\big)^{1\over k} -   \frac{ 1  }{ \frac{  k ^{k-1} }{(k-1)^{k-1}} \big((k-1)( \xi-1)\big) ^{k-1\over k}   }  \\\nonumber
=& \big((k-1)( \xi-1)\big)^{1\over k}  \left(1+ \frac{1}{k-1}\left(1-   \frac{ (k-1)^{k-1}    }{  ( \xi-1)k^{k-1 }  } \right) \right)\\\nonumber
>& \big((k-1)( \xi-1)\big)^{1\over k}.
 \end{align}

We now assume that $n \geq k+1$.
By the connectedness of $\h$ and the choice of $\xi$,
we have $2 \leq \Delta \leq \xi$.
For every edge $e=\{u,v_2, \ldots, v_k\}\in E_\h(u)$,
write $V_1(e)=\{u\}$ and let $V_i(e)=\{u,v_2, \ldots, v_i\}$ for $i=2,\ldots,k$.
Note that for every $i \in [k-1]$,
we have $\Delta(\h-V_{i}(e))\leq \Delta $ and $$d_{\h-V_{i}(e)}(v_{i+1})<d_{\h}(v_{i+1})  \leq  \Delta\leq  \xi,$$
which implies that we can apply the induction hypothesis to the component of $\h-V_{i}(e)$ containing the vertex $v_{i+1}$.
Combining this and Lemma~\ref{basiclemma}(1), we further derive that for every $i \in [k-1]$,
$$
\frac{\mu(\h-V_{i}(e),x) }{\mu(\h-V_{i+1}(e),x)}
>  \big((k-1)(\xi-1)\big)^{1\over k}.
$$
Thus, for every $e\in E_\h(u)$,
$$
\frac{\mu(\h-u,x)}{\mu(\h-e, x)}
=\prod_{i=1}^{k-1}\frac{\mu(\h-V_{i}(e),x) } {\mu(\h-V_{i+1}(e),x)}
>   \big((k-1)(\xi-1)\big)^{k-1\over k}.
$$
Now, combining Lemma \ref{basiclemma}(2), the above inequality, and the assumption $d_{\h}(u) \leq \xi-1$, one may check that
\begin{align}\nonumber
  \frac{\mu(\h, x)}{\mu(\h-u,x)}
  =& x- \sum_{e\in E_\h(u)}\frac{\mu(\h-e, x)}{\mu(\h-u,x)}\\\nonumber
 > & \frac{k}{k-1}\big((k-1)(\xi-1)\big)^{1\over k}-\frac{\xi-1}{ { \big((k-1)(\xi-1)\big)^{ k-1 \over k } }}  \\\nonumber
 = &  \big((k-1)(\xi-1)\big)^{ 1\over k }.
 \end{align}
This completes the proof.
\end{proof}

\begin{theorem}\label{maintheoremsec4}
Let $\h$ be a connected $k$-graph with maximum degree ${\Delta}\geq  2$.
Then $$\lambda(\h)< \frac{k}{k-1}\big((k-1)(\Delta-1)\big)^{1\over k}.$$
\end{theorem}

\begin{proof}
Let $\h$ be a connected $k$-graph with maximum degree ${\Delta}\geq  2$. By Theorem \ref{maintheorem11}, $\lambda(\h)$ is the largest real zero of $\mu(\h, x)$, so
it suffices to show that $\mu(\h, x)> 0$ whenever $x\geq \frac{k}{k-1}\big((k-1)(\Delta-1)\big)^{1\over k}$.
We prove it by induction on $n$.
Let $u$ be a vertex of $\h$ with $d_\h(u)=\delta(\h)$, and we always assume that $x\geq \frac{k}{k-1}\big((k-1)(\Delta-1)\big)^{1\over k}$ in this proof.

If $n=k+1$, then $\h$ consists of two edges sharing $k-1$ vertices by the connectedness of $\h$.
In this case, we have $\Delta=2$,
$d_\h(u)=\delta(\h)=1$,
and $x^k \geq \frac{k^k}{(k-1)^{k-1}}>2.$
So we have
$$\mu(\h-u, x)=x^k-1>1.$$
Moreover, by Lemma \ref{basiclemma}(2),
one may check that
$$\frac{\mu(\h, x)}{\mu(\h-u,x)}
=x-\sum_{e\in E_\h(u)} \frac{\mu(\h-e, x)}{\mu(\h-u,x)}
=x- \frac{x}{x^k-1}
=x\left(1-\frac{1}{x^k-1}\right)>0.$$
The above two inequalities suggest that
$\mu(\h, x)> 0$ whenever $x\geq \frac{k}{k-1}\big((k-1)(\Delta-1)\big)^{1\over k}$,
so the base case of the induction holds.

Assume that $n>k+1$.
For every edge $e=\{u,v_2, \ldots, v_k\}\in E_\h(u)$,
write $V_1(e)=\{u\}$ and let $V_i(e)=\{u,v_2, \ldots, v_i\}$ for $i=2,\ldots,k$.
For every $i \in [k-1]$,
observe that $d_{\h-V_{i}(e)}(v_{i+1})< d_{\h}(v_{i+1})\leq \Delta$.
Thus, we may apply Lemma~\ref{lemma}, with choosing $\xi=\Delta \geq \max\{\Delta(\h-V_{i}(e)),2\}$, to the component of $\h-V_{i}(e)$ containing the vertex $v_{i+1}$.
Combining this and Lemma~\ref{basiclemma}(1),
 we further obtain that for every $i \in [k-1]$,
$$
\frac{\mu(\h-V_{i}(e),x) }{\mu(\h-V_{i+1}(e),x)}
> \big((k-1)(\Delta-1)\big)^{1\over k}.
$$
Thus, for every edge $e\in E_\h(u)$,
$$
\frac{\mu(\h-u,x)}{\mu(\h-e, x)}
=\prod_{i=1}^{k-1}  \frac{\mu(\h-V_{i}(e),x) }{\mu(\h-V_{i+1}(e),x)}
> \big((k-1)(\Delta-1)\big)^{k-1\over k}.
$$
Combining Lemma~\ref{basiclemma}(1), the above inequality, and the fact that $d_\h(u)=\delta(\h)\leq \Delta$, one may check that
\begin{eqnarray*}
 \frac{\mu(\h, x)}{\mu(\h-u,x)}
  &=&x-\sum_{e\in E_\h(u)} \frac{\mu(\h-e, x)}{\mu(\h-u,x)} \\
  &>& \frac{k}{k-1}\big((k-1)(\Delta-1)\big)^{1\over k}-\frac{\Delta}{ \big((k-1)(\Delta-1)\big)^{k-1 \over k} }\\
  &=& \frac{k(\Delta-1)-    \Delta }{ \big((k-1)(\Delta-1)\big)^{k-1\over k} }\\
  &\geq& 0.
\end{eqnarray*}
Using it, to show $\mu(\h, x)> 0$,
it suffices to prove that $\mu(\h-u,x)>0$.
By Lemma~\ref{basiclemma}(1), we need to prove that
$\mu(\G,x)>0$ for every component $\G$ of $\h-u$.
Given a component $\G$ of $\h-u$.
If $\Delta(\G) \geq 2$, then $\mu(\G,x)>0$ follows from the fact that $\Delta(\G) \leq\Delta(\h)$ and the induction hypothesis.
If $\Delta(\G) = 1$, then $\G$ is the $k$-graph consisting of a single edge,
and hence that
$\mu(\G,x)=x^k-1>0$ since $x\geq \frac{k}{k-1}\big((k-1)(\Delta-1)\big)^{1\over k}>1.$
Finally, if $\Delta(\G) = 0$, then $\G$ is the $k$-graph consisting of a single isolated vertex,
and hence that $\mu(\G,x)=x >0$.
This completes the proof of the induction step and establishes the result.
\end{proof}

We now have all the tools to prove Theorem~\ref{maintheorem22} and give a new proof of Theorem~\ref{Lupperbound}.

\begin{proof} [\proofname{ of \bf Theorem~\ref{maintheorem22}.}]
Let $\h$ be a connected $k$-graph with maximum degree ${\Delta}\geq  2$.
Theorem~\ref{zerosymmetricmp} states that the cyclic index of $\mu(\h, x)$ is $k$,
and Theorem~\ref{maintheorem11} implies that $\lambda(\h)$ is a simple root of $\mu(\h,x)$.
Finally, the inequality~\eqref{mr} follows from Corollary~\ref{coro22} and Theorem~\ref{maintheoremsec4}.
The proof is completed.
\end{proof}

\begin{proof}[\proofname{ of \bf Theorem~\ref{Lupperbound}.}]
Let $\T$ be a $k$-tree with maximum degree $\mathnormal{\Delta}\geq  2$.
Then Lemma~\ref{ktreethreeequal} states that $\rho(\T)=\lambda(\T)$,
and the result follows from the upper bound of Theorem~\ref{maintheorem22}.
\end{proof}

%
%

\section{Concluding Remarks}\label{ConcludingRemarks}

In this paper, we present a fundamental characterization of the distribution of the
 zeros of the matching polynomials of $k$-graphs
and generalize some results on the classical matching polynomial to $k$-graphs.
Note that most of the results in this paper can be extended to the multivariate weighted $k$-graphs, the $k$-graph $\h=(V,E)$ associated with an edge-weighted function $\mathbf{w}: E\rightarrow \C$ and a vertex-indeterminate $\mathbf{x}=\{x_v\}_{v\in V}$, with some appropriate adjustment.
For the sake of simplicity, we chose not to pursue that direction in detail.

There is another interesting function related to the matching polynomial,
the {\it matching generating function} of a $k$-graph $\h$, which is defined by
$$m(\h,x)=\sum_{r\geq 0} p(\h,r) x^r.$$
Note that
$$\mu(\h, x)=\sum_{r\geq 0}(-1)^rp(\h,r) x^{|V(\h)|-kr}=x^{|V(\h)|}\sum_{r\geq 0}p(\h,r) (-x^{-k})^r,$$
so we have
$$
\mu(\h,x)=x^{|V(\h)|}m(\h,-x^{-k}).
$$
Therefore, we can obtain some results similar to Theorem~\ref{d2} and Theorem~\ref{maintheoremsec4} for the matching generating function.

As mentioned in Section~1, the result of Li and Mohar~\cite{LiMohar} indicates that  for a connected $k$-graph $\h$ with maximum degree $\Delta$, the threshold bound  $$\frac{k}{k-1}\big((k-1)(\Delta-1)\big)^{1\over k}$$
plays an important role in studying the second  eigenvalue of $\h$.
Besides, Theorem~\ref{Lupperbound} states that this value is exactly an upper bound of the spectral radius of a $k$-tree with maximum degree $\Delta \geq 2$.
In fact, by combining Theorem~\ref{maintheorem11} and Theorem~\ref{Lupperbound},
we may obtain another proof for the upper bound of Theorem~\ref{maintheorem22}.
Therefore, in the current setting, Theorem~\ref{maintheorem22} can be viewed as a new version of Theorem~\ref{Lupperbound} from the view point of matching polynomials.
 The main idea of~\cite{Marcus} seems to imply that the former is more essential than the latter in the study of the second eigenvalues of hypergraphs and Ramanujan hypergraphs.

A sequence $a_0, a_1,\ldots, a_n$, of real numbers
is said to be {\it logarithmically concave} (or log-concave for short) if $a_i^2\geq a_{i-1}a_{i+1}$ for all $1 \leq i\leq n- 1$.
Many important sequences in combinatorics are known to be log-concave. We refer the reader to a survey by Stanley~\cite{Stanley89} for various examples and more background.
Applying  the rooted-rootedness of the matching polynomial in Theorem~\ref{Heilmanntheorem}, Heilmann and Lieb~\cite{Heilmann} prove that the matching number sequence $\{p(G,r)\}_{r\geq 0}$ of a graph $G$ is log-concave.
However, for $k\geq 3$, the real-rootedness for matching polynomials of $k$-graphs is invalid as proved in Theorem~\ref{maintheorem22}. Thus, it would be interesting to study the log-concave property of the matching number sequence of a $k$-graph.

\section*{Acknowledgements}
The authors would like to thank Hyunwoo Lee for useful comments.


\end{document}